\documentclass[11pt,english]{amsart}
\usepackage[T1]{fontenc}
\usepackage[utf8]{luainputenc}
\usepackage{geometry}
\usepackage{babel}
\usepackage{amsthm}
\usepackage{amstext}
\usepackage{amssymb}
\usepackage{esint}
\usepackage[numbers]{natbib}
\usepackage[unicode=true,pdfusetitle,
 bookmarks=true,bookmarksnumbered=false,bookmarksopen=false,
 breaklinks=false,pdfborder={0 0 1},backref=false,colorlinks=false]
 {hyperref}

\makeatletter
\numberwithin{equation}{section}
\numberwithin{figure}{section}
  \theoremstyle{plain}
  \newtheorem{assumption}{\protect\assumptionname}
\theoremstyle{plain}
\newtheorem{thm}{\protect\theoremname}
  \theoremstyle{remark}
  \newtheorem{rem}{\protect\remarkname}
  \theoremstyle{plain}
  \newtheorem{lem}{\protect\lemmaname}
  \theoremstyle{definition}
  \newtheorem{defn}{\protect\definitionname}

\makeatother

  \providecommand{\assumptionname}{Assumption}
  \providecommand{\definitionname}{Definition}
  \providecommand{\lemmaname}{Lemma}
  \providecommand{\remarkname}{Remark}
\providecommand{\theoremname}{Theorem}

\begin{document}

\title{On The Shape of The Free Boundary of Variational Inequalities with
Gradient Constraints}

\author{Mohammad Safdari}
\begin{abstract}
In this paper we derive an estimate on the number of local maxima
of the free boundary of the minimizer of 
\[
I(v):=\int_{U}\frac{1}{2}|Dv|^{2}-\eta v\, dx,
\]
subject to the pointwise gradient constraint 
\[
|Dv|_{p}\le1.
\]
This also gives an estimate on the number of connected components
of the free boundary.%
\thanks{School of Mathematics, Institute for Research in Fundamental Sciences
(IPM), P.O. Box: 19395-5746, Tehran, Iran\protect \\
Email address: safdari@ipm.ir%
}
\end{abstract}

\maketitle

\section{Introduction}

Let $U\subset\mathbb{R}^{2}$ be a simply connected bounded open set
whose boundary is a simple closed Jordan curve consisting of arcs
$S_{1},\cdots,S_{m}$ that are $C^{k,\alpha}$ $(k\ge3\,,\,\alpha>0)$
or analytic up to their endpoints, satisfying Assumption \ref{assu: 1}
below. We denote by $V_{i}:=\bar{S}_{i}\cap\bar{S}_{i+1}$ the vertices
of $\partial U$, and we assume that all the vertices are \textbf{nonreentrant}
corners i.e. their opening angle is less than $\pi$.

Let 
\begin{equation}
I[v]:=\int_{U}\frac{1}{2}|Dv|^{2}-\eta v\, dx,
\end{equation}
with $\eta>0$. Let $u$ be the minimizer of $I$ over 
\begin{equation}
K:=\{v\in H_{0}^{1}(U)\,\mid\,\gamma_{q}(Dv)\leq1\textrm{ a.e. }\}.
\end{equation}
Where $\gamma_{q}$ is the $q$-norm on $\mathbb{R}^{2}$ 
\[
\gamma_{q}((x_{1},x_{2})):=(|x_{1}|^{q}+|x_{2}|^{q})^{\frac{1}{q}}.
\]
As showed in \citet{MR1} we know that $u\ge0$ and it is also the
minimizer of $I$ over 
\begin{equation}
\tilde{K}:=\{v\in H_{0}^{1}(U)\,\mid\, v(x)\leq d_{p}(x,\partial U)\textrm{ a.e. }\}.
\end{equation}
Here $p=\frac{q}{q-1}$ is the dual exponent to $q$, and $d_{p}$
is the metric associated to $\gamma_{p}$. We also assume that $1<q\le2$,
so $2\le p<\infty$.

When $p=q=2$, in which case $\gamma_{2}$ is the Euclidean norm,
the above problem is the famous elastic-plastic torsion problem. The
regularity and the shape of the free boundary of the elastic-plastic
torsion problem is studied by \citet{MR0412940,MR0521411}, \citet{MR534111},
\citet{MR552267}, and \citet{MR563207}.

In \citet{MR1,Safdari20151}, we extended some of their results to
the more general problem explained above. In this work, we continue
this study and generalize some other parts of the above works. Especially,
we extend the results in \citet{MR552267}, and some of the reflection
methods in \citet{MR563207}.

A motivation for our study was to fill the gap between the known regularity
results mentioned above, and the still open question of regularity
of the minimizer of some convex functionals subject to gradient constraints
arising in random surfaces. To learn about the latter, see the work
of \citet{MR2605868}.

Let us summarize here some of the results proved in \citet{MR1,Safdari20151}.
It has been proved that $u\in C_{\textrm{loc}}^{1,1}(U)=W_{\textrm{loc}}^{2,\infty}(U)$.
Also, we have the equalities 
\[
E:=\{x\in U\,\mid\, u(x)<d_{p}(x,\partial U)\}=\{x\in U\,\mid\,\gamma_{q}(Du(x))<1\},
\]
and 
\[
P:=\{x\in U\,\mid\, u(x)=d_{p}(x,\partial U)\}=\{x\in U\,\mid\,\gamma_{q}(Du(x))=1\}.
\]
The first region is called the \textbf{elastic} region and the second
one is called the \textbf{plastic} region. It is easy to see that
if $x\in P$ and $y\in\partial U$ is one of the $p$-closest points
to $x$ on the boundary, then the segment between $x$ and $y$ (which
is obviously in $U$) lies inside $P$. In addition, we have $\Delta u=-\eta$
over $E$,%
\footnote{This implies that $u>0$ in $U$, since $u>0$ in $P$, and by the
strong maximum principle $u>0$ in $E$ too. Note that $u$ can not
vanish identically over $E$ due to the equation $\Delta u=-\eta\ne0$.%
} and $\Delta u\ge-\eta$ a.e. over $U$.

The complement of the largest open set over which $d_{p}(x):=d_{p}(x,\partial U)$
is $C^{1,1}$, is called the \textbf{$p$-ridge} and is denoted by
$R_{p}$. It has been shown that $R_{p}$ consists of those points
in $U$ with more than one $p$-closest point on $\partial U$, and
those other points $x$ at which $d_{p}(x)=\frac{1}{\kappa_{p}(y)}$
(we define $\kappa_{p}(y)$ below). One nice property of the $p$-ridge
is that the $p$-closest point on $\partial U$ varies continuously
in $\bar{U}-R_{p}$. Also, $R_{p}\subset E$, and outside $R_{p}$,
$d_{p}$ is as smooth as $\partial U$, provided that $\partial U$
satisfies 
\begin{assumption}
\label{assu: 1}We assume that at the points where the normal to one
of the $S_{i}$'s is parallel to one of the coordinate axes, the curvature
of $S_{i}$ is small. In the sense that, if we have $(s+a_{0},b(s))$
as a nondegenerate $C^{k,\alpha}$ $(k\geq3\,,\;0<\alpha<1)$ parametrization
of $S_{i}$ around $y_{0}:=(a_{0},b(0))$, and $b'(0)=0$; then we
assume $b'$ goes fast enough to $0$ so that $b'(s)=c(s)|c(s)|^{p-2}$,
where $c(0)=0$, and $c$ is $C^{k-1,\alpha}$. Note that $y_{0}$
can be one of the endpoints of $S_{i}$.

Also we require $c'(0)$ to be small enough so that $1-c'(0)d_{p}(\cdot)$
does not vanish at the points inside $U$ that have $y_{0}$ as the
only $p\hspace{1bp}$-closest point on $\partial U$.
\end{assumption}

It is easy to show that there is a $p$-circle inside $U$ that touches
$\partial U$ only at $y_{0}$ (see the proof of Theorem \ref{thm: ridge is inside}
below). We will call these points the degenerate points of Assumption
\ref{assu: 1}. Note that we modified this assumption to be slightly
different than what appeared in \citet{Safdari20151}, to emphasize
that we require this assumption to also hold at the endpoints of the
arcs $S_{1},\cdots,S_{m}$. 

Away from the $p$-ridge we have 
\begin{equation}
\Delta d_{p}(x)=\frac{-(p-1)\tau_{p}(y)\kappa_{p}(y)}{1-\kappa_{p}(y)d_{p}(x)}.\label{eq: Lap d_p}
\end{equation}
Here $y\in\partial U$ is the $p$-closest point to $x$, and if $(a(\cdot),b(\cdot))$
is a parametrization of $\partial U$ around $y$, 
\[
\kappa_{p}:=\frac{a'b''-b'a''}{(p-1)|a'|^{\frac{p-2}{p-1}}|b'|^{\frac{p-2}{p-1}}(|a'|^{\frac{p}{p-1}}+|b'|^{\frac{p}{p-1}})^{\frac{p+1}{p}}}
\]
is the \textbf{$p$-curvature}, and
\[
\tau_{p}:=\frac{(|a'|^{\frac{2}{p-1}}+|b'|^{\frac{2}{p-1}})|a'b'|^{\frac{p-2}{p-1}}}{(|a'|^{\frac{p}{p-1}}+|b'|^{\frac{p}{p-1}})^{\frac{2p-2}{p}}}
\]
is another reparametrization invariant quantity. Note that at the
degenerate points of Assumption \ref{assu: 1}, we have%
\footnote{If we can approach them with nondegenerate points.%
} $\lim\kappa_{p}=c'(0)$ and $\tau_{p}=0$. Let us also record here
that outside $R_{p}$, $y$ is a $C^{k-1,\alpha}$ function of $x$,
and 
\begin{equation}
Dd_{p}(x)=\frac{\nu(y)}{\gamma_{q}(\nu(y))},\label{eq: Dd_p}
\end{equation}
where $\nu$ is the inward normal to $\partial U$. Note that nonreentrant
corners can not be the $p$-closest point on $\partial U$ to any
point inside $U$.

Let $y=f(s)$ $(0\le s\le L)$ be a parametrization of $\partial U$.
Then it has been proved that the \textbf{free boundary}, $\Gamma:=\partial E\cap U$,
can be parametrized by $f(s)+\delta(s)\mu(s)$. Here $\delta:[0,L]\to\mathbb{R}$
is a continuous and nonnegative function, and $\mu(s)$ is the unique
direction at $f(s)$ along which points inside $U$ have $f(s)$ as
the $p$-closest point on $\partial U$. $\mu$ is called the inward
\textbf{$p$-normal}, and is given by the formula 
\begin{equation}
\mu:=\frac{1}{(|\nu_{1}|^{\frac{p}{p-1}}+|\nu_{2}|^{\frac{p}{p-1}})^{\frac{1}{p}}}(\textrm{sgn}(\nu_{1})|\nu_{1}|^{\frac{1}{p-1}},\textrm{sgn}(\nu_{2})|\nu_{2}|^{\frac{1}{p-1}}),\label{eq: p-normal}
\end{equation}
where as before $\nu=(\nu_{1},\nu_{2})$ is the inward normal to $\partial U$
at $f(s)$. Furthermore, we know that $\Gamma$ is a smooth curve
with no cusp as smooth as the tangent bundle of $\partial U$. Also,
$\delta\equiv0$ in a neighborhood of nonreentrant corners, since
it has been shown that nonreentrant corners have an elastic neighborhood
in $U$. Note that on $\Gamma$ we have $u=d_{p}$ and $Du=Dd_{p}$.

Also note that the above characterization of the free boundary implies
that $E$ is a simply connected domain bounded by a simple closed
Jordan curve.

Let us give a global regularity result not mentioned in \citet{MR1,Safdari20151}.
\begin{thm}
\label{thm: regularity}When all the vertices of $\partial U$ are
nonreentrant corners, we have $u\in C^{1,\alpha}(\bar{U})$ for some
$\alpha>0$. If $\partial U$ has no corners, the conclusion holds
for all $\alpha\in(0,1)$.\end{thm}
\begin{proof}
Note that by the gradient constraint we have $u\in W^{1,\infty}(U)=C^{0,1}(\bar{U})$.
Furthermore 
\[
\Delta u=\begin{cases}
-\eta & \textrm{ in }E\\
\Delta d_{p} & \textrm{ a.e. in }P.
\end{cases}
\]
Also note that by Assumption \ref{assu: 1}, $\kappa_{p}$ is bounded
on $\partial U$. Thus $1-\kappa_{p}d_{p}\to1$ uniformly, as we approach
$\partial U$. Also, $1-\kappa_{p}d_{p}>0$ on $P$,%
\footnote{If $1-\kappa_{p}d_{p}<0$ at some point in $P$, moving toward $\partial U$
along the segment that connects that point to its $p$-closest point
on $\partial U$, we find a point in $P$ at which $1-\kappa_{p}d_{p}=0$,
as $d_{p}$ changes linearly along that segment.%
} so it has a positive minimum there. In addition, $\tau_{p}\kappa_{p}$
is bounded on $\partial U$ as we assumed that $S_{i}$'s are smooth
up to their endpoints. Hence $\Delta d_{p}$ is bounded on $P$.
Thus $\Delta u$ is bounded there too. Therefore as $u\in C^{0}(\bar{U})$,
we can apply the Calderon-Zygmund estimate and conclude that $u$
is in $W^{2,s}$ for any $s\in(1,\infty)$, around any $C^{1,1}$
portion of $\partial U$. Thus $u$ is in $C^{1,\alpha}$ around points
in the interior of $S_{i}$'s, for any $\alpha\in(0,1)$. (Consult
Theorem 9.15 in \citet{MR1814364}. Note that we need to multiply
$u$ by a smooth bump function with support around some smooth part
of $\partial U$, and use the fact that $u,Du,\Delta u$ are bounded.)
As nonreentrant corners have an elastic neighborhood in $U$, around
them we have $\Delta u=-\eta$. Now as $u$ vanishes on $\partial U$,
we can apply the results in \citet{MR576102} to deduce that $u$
is in $C^{1,\alpha}$ for some $\alpha>0$ around these corners.
\end{proof}
Let us also give an interesting consequence of Assumption \ref{assu: 1}. 
\begin{thm}
\label{thm: ridge is inside}Every smooth point of $\partial U$ has
a $U$-neighborhood that does not intersect $R_{p}$. \end{thm}
\begin{proof}
The reason is that, locally around smooth points, $\partial U$ has
\textit{uniform interior $p$-circle property}. This means that for
any smooth point $y_{0}\in\partial U$ and any $y\in\partial U$ close
enough to $y_{0}$, there is a $p$-circle inside $U$ whose boundary
touches $\partial U$ only at $y$, and its $p$-radius is independent
of $y$. This implies that close to $y_{0}$, no point of $U$ has
more than one $p$-closest point on $\partial U$. Also, as $\kappa_{p}$
is bounded on $\partial U$ by Assumption \ref{assu: 1} and smoothness
of $S_{i}$'s up to their endpoints, $1-\kappa_{p}d_{p}\ne0$ near
the boundary. Thus we get the result.

To prove the property, first assume that $y_{0}=(a_{0},b_{0})$ is
a degenerate point of Assumption \ref{assu: 1}, and around it we
can parametrize $\partial U$ by 
\[
s\mapsto(s+a_{0},b(s)).
\]
Where $b(0)=b_{0}$, $b'(0)=0$, and $b(s)=c(s)|c(s)|^{p-2}$ for
some smooth enough function $c$. We assume that $U$ is above $\partial U$
around $y_{0}$. 

Let $s_{1}$ be close to $0$, and consider $y_{1}=(s_{1}+a_{0},b(s_{1}))$
near $y_{0}$. Then, $\frac{(-c(s_{1}),1)}{(1+|c(s_{1})|^{p})^{\frac{1}{p}}}$
is the $p$-normal at $y_{1}$. Consider the $p$-circle with $p$-radius
$r$ and center $(a_{1}+a_{0},b_{1})$, where $a_{1}:=s_{1}-\frac{rc(s_{1})}{(1+|c(s_{1})|^{p})^{\frac{1}{p}}}$
and $b_{1}:=b(s_{1})+\frac{r}{(1+|c(s_{1})|^{p})^{\frac{1}{p}}}$.
We will show that this $p$-circle which passes through $y_{1}$,
is above $\partial U$ near $y_{1}$. Let 
\[
\alpha(s):=-(r^{p}-|s-a_{1}|^{p})^{\frac{1}{p}}+b_{1}-b(s).
\]
It is enough to show that $\alpha$ is positive around $s_{1}$. Note
that $\alpha(s_{1})=0$.

For this to happen, it suffices to show that 
\[
\alpha'(s)=\frac{(s-a_{1})|s-a_{1}|^{p-2}}{(r^{p}-|s-a_{1}|^{p})^{\frac{p-1}{p}}}-c(s)|c(s)|^{p-2}
\]
is positive after $s_{1}$ and negative before it. But as the map
$s\mapsto s|s|^{p-2}$ is increasing, we just need to show that 
\[
\beta(s):=\frac{s-a_{1}}{(r^{p}-|s-a_{1}|^{p})^{\frac{1}{p}}}-c(s)
\]
has the same property. As $\beta(s_{1})=0$, it is sufficient to show
that 
\[
\beta'(s)=\frac{1}{(r^{p}-|s-a_{1}|^{p})^{\frac{1}{p}}}+\frac{|s-a_{1}|^{p}}{(r^{p}-|s-a_{1}|^{p})^{\frac{p+1}{p}}}-c'(s)
\]
is positive. 

Choose $r$ small enough so that $c'(s)<\frac{1}{2r}$ for $|s|\le2\epsilon_{0}$,
where $\epsilon_{0}$ is very small compared to $r$. Then for any
$s_{1}$ with $|s_{1}|<\epsilon_{0}$, we have $\beta'(s)>\frac{1}{2r}$
for $|s-s_{1}|\le\epsilon_{0}$. Thus $\alpha(s)>0$ for $0<|s-s_{1}|\le\epsilon_{0}$.
Now inside the $p$-circle with $p$-radius $r$, we take a $p$-circle
with $p$-radius $r_{1}$ that passes through $y_{1}$. Let $|s_{1}|\le\frac{1}{4}\epsilon_{0}$.
We can take $r_{1}$ to be small enough (independently from $y_{1}$),
so that this smaller $p$-circle has a positive distance from $\partial U-\{(s+a_{0},b(s))\,\mid\,|s|<\frac{1}{2}\epsilon_{0}\}$.
Hence the smaller $p$-circle is inside $U$, and this is what we
wanted to prove.

Now assume that $y_{0}$ is a nondegenerate point. Then due to the
inverse function theorem, we can find a parametrization for $\partial U$
around $y_{0}$ of the form
\[
s\mapsto(s+a_{0},b(s)).
\]
This time $b'(s)\ne0$ for $s$ small, so we can define the smooth
function $c(s):=\frac{b'(s)}{|b'(s)|}|b'(s)|^{\frac{1}{p-1}}$. Hence
$b'=c|c|^{p-2}$ and we can repeat the above argument.\end{proof}
\begin{rem}
When $p\ne2$, this theorem is false without Assumption \ref{assu: 1}.
A simple example is a disk, whose $p$-ridge is the union of its two
diagonals parallel to the coordinate axes.
\end{rem}

\begin{rem}
An important consequence of this theorem is that $d_{p}$ is at least
$C^{1}$ up to smooth points of $\partial U$. The reason is that
$d_{p},Dd_{p}$ are uniformly continuous on a $U$-neighborhood of
these points. 
\end{rem}
Now let us briefly comment on the case that some vertices $V_{i}$
are \textbf{reentrant} corners, i.e. their opening angle is greater
than $\pi$. The main difference that these corners have with nonreentrant
ones, is that they are the $p$-closest point on $\partial U$ to
some points inside $U$. In fact, if we denote by $\mu_{i1},\mu_{i2}$
the inward $p$-normals to respectively $S_{i},S_{i+1}$ at $V_{i}$,
then the points in $U$ between $\mu_{i1},\mu_{i2}$ and close to
$V_{i}$ have $V_{i}$ as the only $p$-closest point on $\partial U$.
We denote this set of points by $U_{i}$. Note that $U_{i}$ is an
open subset of $U$.

It is obvious that $d_{p}$ is analytic on $U_{i}$. The $p$-ridge
is characterized as before, and is inside the elastic region. The
other difference is that on $\mu_{ij}$, we can only say that $d_{p}$
is $C^{1,1}$ at the points where $d_{p}\ne\frac{1}{\kappa_{p}}$.
Furthermore, the free boundary is an analytic curve inside $U_{i}$.
Here, $\delta$ is a function of the angle between $\mu_{i1}$ and
the segment connecting $V_{i}$ to the free boundary. Also, Theorem
\ref{thm: regularity} does not hold when some of the vertices are
reentrant corners. Although $Du$ remains bounded as we approach a
reentrant corner, it is not necessarily $C^{1}$ there.

\section{Flat Boundaries}

We start with a lemma about the level sets of a function satisfying
an elliptic equation in some region of the plane.
\begin{lem}
\label{lem: level sets}Let $U\subset\mathbb{R}^{2}$ be a bounded
simply connected domain whose boundary is a simple closed Jordan curve.
Suppose $u\in C^{2}(U)\cap C(\bar{U})$ is a nonconstant function
satisfying 
\[
Lu:=-a^{ij}D_{ij}^{2}u+b^{i}D_{i}u=0.
\]
Where $L$ is a uniformly elliptic operator with continuous coefficients.
Then the closure of every level set of $u$ in $U$, intersects $\partial U$. 

Furthermore, when $L=-\Delta$, the closure of every connected component
of any level set of $u$ in $U$, intersects $\partial U$.\end{lem}
\begin{proof}
Let $S:=\{x\in U\,\mid\, u(x)=c\}$ be a nonempty level set, and suppose
to the contrary that $\bar{S}\subset U$. Then as both $\bar{S}$
and $\partial U$ are compact, their distance, $2\varepsilon$, is
positive. For any $y\in\partial U$, let $U_{\varepsilon}(y)$ be
the connected component of $B_{\varepsilon}(y)\cap U$ that has $y$
on its boundary. First, note that there is at most one such component
since $\partial U$ is a simple Jordan curve. Second, on any $U_{\varepsilon}(y)$,
$u$ is either greater than $c$ or less than $c$. The reason is
that if both happen, $u$ must take the value $c$ in $U_{\varepsilon}(y)$
which is impossible.

Now suppose that for some $y_{0}\in\partial U$ we have $u<c$ on
$U_{\varepsilon}(y_{0})$. We claim that the same thing happens for
every $y\in\partial U$. Let 
\[
A:=\{y\in\partial U\,\mid\, u<c\textrm{ on }U_{\varepsilon}(y)\}.
\]
Obviously $A$ is open in $\partial U$. But it is also closed, since
if for $y\in\partial U$ we have $y_{i}\to y$ for some sequence $y_{i}\in A$,
then for large enough $i$ we have $y\in\overline{U_{\varepsilon}(y_{i})}$.
Thus as $A$ is nonempty and $\partial U$ is connected we have $A=\partial U$.
This implies that $u\le c$ on $\partial U$. But in that case, the
strong maximum principle implies that $u$ is constant, which is a
contradiction.

Now suppose $L=-\Delta$. Then $u$ is analytic inside $U$, and its
level sets are locally, several analytic arcs emanating from a point.
Suppose, $S_{1}\subset U$ is a connected component of $S$, and $\bar{S}_{1}\subset U$.
Then as $S_{1}$ is a maximal connected subset of $S$, we have $S_{1}=\bar{S}_{1}$.
Thus $S_{1}$ is compact. Hence $S_{1}$ has a positive distance from
$\partial U$. It also has a positive distance from $\overline{S-S_{1}}$.
The reason is that if $S_{1}\cap\overline{S-S_{1}}\ne\emptyset$,
then there is a sequence in $S-S_{1}$ converging to a point in $S_{1}$,
which is also in $U$. But this implies that, that sequence belongs
to one of the analytic arcs emanating from that point. This means,
that sequence belongs to $S_{1}$, which is a contradiction.

Therefore, we can enclose $S_{1}$ by a simple closed Jordan curve
inside $U$ that still has a positive distance from $\bar{S}_{1}$,
and leaves $\overline{S-S_{1}}$ outside. We can argue as before and
get a contradiction, noting that as $u$ is analytic, it can not be
constant on this new domain.
\end{proof}

\begin{defn}
If $\delta(s)>0$ for $s\in(a,b)$ and $\delta(a)=\delta(b)=0$ then
we call the set 
\[
\{y(s)+t\mu(s)\,\mid\, s\in[a,b]\,,\, t\in[0,\delta(s)]\}
\]
a \textbf{plastic component}.
\end{defn}
Note that there are at most countably many plastic components. The
following theorem is a stronger version of a result proved in \citet{Safdari20151}.
Here, we also give some details of the proof that are not presented
there. For this theorem, we can allow $U$ to have several holes homeomorphic
to a disk, and not be simply connected.
\begin{thm}
\label{thm: plastic component}The number of plastic components attached
to a closed line segment of $\partial U$ is finite, if the endpoints
of the segment are not reentrant corners, and a neighborhood of each
endpoint in the segment has an elastic neighborhood in $U$, or belongs
to a plastic component.\end{thm}
\begin{proof}
Let the line segment be 
\[
\lambda_{1}:=\{(x_{1},\rho_{1}x_{1}+\rho_{2})\,\mid\, a\leq x_{1}\leq b\},
\]
and assume that $U$ is above the segment. Suppose to the contrary
that there are infinitely many plastic components 
\[
P_{i}=\{(x_{1},\rho_{1}x_{1}+\rho_{2})+t(\mu_{1},\mu_{2})\,\mid\, x_{1}\in[a_{i},b_{i}]\,,\, t\in[0,\delta(x_{1})]\}
\]
attached to the line segment. Where $\mu:=(\mu_{1},\mu_{2})$ is the
inward $p$-normal, $b_{i}\leq a_{i+1}$, and as noted before $\delta$
is a continuous nonnegative function on $[a,b]$. Let 
\[
H_{i}:=\underset{x\in[a_{i},b_{i}]}{\max}\delta(x).
\]
Since $b_{i}-a_{i}\rightarrow0$ as $i\rightarrow\infty$, we must
have $H_{i}\rightarrow0$. Otherwise a subsequence, $H_{n_{i}}$ converges
to a positive number and by taking a further subsequence we can assume
that this subsequence is $\delta(x_{n_{i}})$ where $x_{n_{i}}\to c$.
But this contradicts the continuity of $\delta$ at $c$ because $b_{n_{i}}\to c$
too. 

Hence any line $x_{2}=\rho_{1}x_{1}+\rho_{2}+\epsilon$ intersects
only a finite number, $n(\epsilon)$, of $P_{i}$'s, and $n(\epsilon)\rightarrow\infty$
as $\epsilon\rightarrow0$. 

Consider the tilted graph of $\delta$ over $\lambda_{1}$. It is
in the subset of $U$ consisting of points whose $p$-closest point
on $\partial U$ belongs to $\lambda_{1}$. Since $U-R_{p}$ is open,
the subset of this part of the tilted graph over which $\delta>0$
has a positive distance from $R_{p}$. On the part where $\delta=0$,
we have the same conclusion, noting that $R_{p}$ has a positive distance
from the interior of $\lambda_{1}$. If $\delta>0$ at the endpoints
of $\lambda_{1}$, we can argue as above, and if $\delta=0$ there,
we actually work with a subsegment of $\lambda_{1}$. Thus as the
$p$-closest point on $\partial U$ varies continuously in $U-R_{p}$,
the $p$-normals to $\lambda_{1}$ are parallel,%
\footnote{This is needed to prove that above the tilted graph of $\delta$,
along the $p$-normal at one of the endpoints, the $p$-closest point
on $\partial U$ is still that endpoint.%
} and $\lambda_{1}$ is compact, the tilted graph of $\delta$ attached
to $\lambda_{1}$ has a tubular neighborhood in $E$ that does not
intersect $R_{p}$ and consists of points whose $p$-closest point
on $\partial U$ belongs to $\lambda_{1}$.

Consider a piecewise analytic curve $\gamma$ in this tubular neighborhood,
that has no self intersection. The endpoints of $\gamma$ are on $\lambda_{1}$.
We specify the left endpoint of $\gamma$, the other one is similar.
If the part of $\lambda_{1}$ near its left endpoint has an elastic
neighborhood, we start $\gamma$ slightly to the right of the left
endpoint, staying in the elastic region. If the part of $\lambda_{1}$
near its left endpoint belongs to a plastic component, we start $\gamma$
at the maximum point on the tilted graph of $\delta$ on that plastic
component, which is on the right of the left endpoint. Even if the
maximum happens at the endpoint itself, we have to start $\gamma$
slightly after the endpoint on the free boundary.%
\footnote{This is necessary to ensure that $D_{\zeta}u$ has analytic continuation
around the endpoint of $\gamma$. Since, although $\partial U$ is
smooth at the endpoint of $\lambda_{1}$, it is not necessarily analytic
there.%
}

By our construction, $\gamma$ is close enough to $\lambda_{1}$ so
that for points between them, the $p\hspace{1bp}$-distance to $\partial U$
is the $p\hspace{1bp}$-distance to $\lambda_{1}$. Thus for those
points $d_{p}(x,\partial U)$ is a function of only $-\rho_{1}x_{1}+x_{2}$.
Since as proved in \citet{Safdari20151}, the $p$-distance to a line
is a multiple of the $2$-distance to the line, with coefficient depending
only on the line and $p$. Thus for $\zeta:=\frac{1}{\sqrt{1+\rho_{1}^{2}}}(1,\rho_{1})$
we have $D_{\zeta}d_{p}=0$ in this region.

Now let $E_{0}$ to be the elastic region enclosed by $\gamma$ and
the tilted graph of $\delta$ over $\lambda_{1}$. Let $\epsilon>0$
be small enough. On every open connected segment of $E_{0}\cap\{x_{2}=\rho_{1}x_{1}+\rho_{2}+\epsilon\}$
with endpoints on the free boundary of two different $P_{j}$'s, the
function $D_{\zeta}(u-d_{p})=D_{\zeta}u$ is analytic and changes
sign, as $u-d_{p}$ is zero on the endpoints and negative between
them. Let 
\[
\tilde{c}_{i}:=(c_{i},\rho_{1}c_{i}+\rho_{2}+\epsilon)\,,\,\tilde{c}_{i+1}:=(c_{i+1},\rho_{1}c_{i+1}+\rho_{2}+\epsilon)
\]
for $c_{i}<c_{i+1}$ be points close to those endpoints such that
\[
D_{\zeta}u(\tilde{c}_{i})<0\;,\; D_{\zeta}u(\tilde{c}_{i+1})>0.
\]
We can also assume that $D_{\zeta}u\le0$ on the part of the segment
joining $\tilde{c}_{i}$ to the free boundary, and similarly $D_{\zeta}u\ge0$
on the part of the segment joining $\tilde{c}_{i+1}$ to the free
boundary. Let $\sigma_{i}(\epsilon)$ be the connected component containing
$\tilde{c}_{i}$, of the level set 
\[
\{y\in E_{0}\,\mid\, D_{\zeta}u(y)=D_{\zeta}u(\tilde{c}_{i})\}.
\]

Then by Lemma \ref{lem: level sets}, the closure of the connected
components of the level sets of the harmonic function $D_{\zeta}u$,
will intersect the boundary of its domain $E_{0}$. Note that $\partial E_{0}$
consists of $\gamma$ and part of the image of 
\[
x_{1}\mapsto(x_{1},\rho_{1}x_{1}+\rho_{2})+\delta(x_{1})(\mu_{1},\mu_{2}),
\]
hence it is a simple closed Jordan curve, and $E_{0}$ is simply connected.%
\footnote{Note that $E_{0}$ is simply connected even when $U$ is not, since
for all points in it, the $p\hspace{1bp}$-closest point on $\partial U$
lies on $\lambda_{1}$. Thus no other part of $\partial U$ can be
inside it.%
} Also as shown in the introduction, $D_{\zeta}u$ is continuous on
$\overline{E_{0}}$, as we are away from reentrant corners. Obviously,
$D_{\zeta}u$ is not constant over $E_{0}$ too, unless the points
$\tilde{c}_{i},\tilde{c}_{i+1}$ do not exist, in which case we have
at most one plastic component.

We claim that there is a path in $\overline{\sigma_{i}(\epsilon)}$
that connects $\tilde{c}_{i}$ to a point on $\gamma$. To see this,
note that $D_{\zeta}u=D_{\zeta}(u-d_{p})$ is zero on the free boundary
and on the segment $\lambda_{1}$. Hence, $\overline{\sigma_{i}(\epsilon)}$
must intersect $\gamma$. In addition, $D_{\zeta}u$ is harmonic on
a neighborhood of $\gamma$. The reason is that locally, $D_{\zeta}u$
has harmonic continuation across the elastic parts of the segment
$\lambda_{1}$, and the free boundary attached to it, since they are
analytic curves and $D_{\zeta}u$ vanishes along them.%
\footnote{Note that the graph of an analytic function can be transformed into
a line segment by a conformal map. Also, a conformal change of variables
takes harmonic functions to harmonic functions. Now, the Schwarz reflection
principle gives the result.%
} Thus $D_{\zeta}u$ is harmonic on a neighborhood of $\overline{\sigma_{i}(\epsilon)}$.
But, the level sets of a harmonic function are locally, the union
of several analytic arcs emanating from a vertex. On the other hand,
as $\gamma$ is piecewise analytic, $\overline{\sigma_{i}(\epsilon)}\cap\gamma$
is a finite set. Hence, $\overline{\sigma_{i}(\epsilon)}$ is locally
path connected, and as it is connected it must be path connected.

Consider an injective%
\footnote{As $\overline{\sigma_{i}(\epsilon)}$ is Hausdorff and path connected,
it is arcwise connected, i.e. any two distinct points in it can be
connected by an injective continuous path.%
} path that connects $\tilde{c}_{i}$ to $\gamma$, and its last intersection
with the segment joining $\tilde{c}_{i}$ to the free boundary along
the line $x_{2}=\rho_{1}x_{1}+\rho_{2}+\epsilon$. Let $\phi_{i}(\epsilon)$
be the union of the part of the path that connects that last intersection
point to $\gamma$, and the part of the segment that joins it to the
free boundary. Therefore, $\phi_{i}(\epsilon)$ is a simple Jordan
curve, connecting two distinct points of $\partial E_{0}$. Hence
it disconnects $E_{0}$. Since obviously $\phi_{i+1}(\epsilon)\cap\phi_{i}(\epsilon)=\emptyset$,
$D_{\zeta}u$ must change sign at least $n(\epsilon)-1$ times along
$\gamma$. But $n(\epsilon)-1$ grows to infinity as $\epsilon\rightarrow0$,
contradicting the fact that $\gamma$ is piecewise analytic and $D_{\zeta}u$
is analytic on a neighborhood of it.\end{proof}
\begin{rem}
The only kind of line segments not covered by the above theorem, are
those that one of their endpoints is the accumulation point of a family
of plastic components. The main difficulty in this case is that, $D_{\zeta}u$
might not have analytic continuation in a neighborhood of the endpoints
of $\gamma$. For these segments, we can still apply the above reasoning
to their proper subsegments. Since we can choose the curve $\gamma$
to start and end slightly before and after the endpoints of the subsegment,
at new endpoints satisfying one of the conditions of the theorem.
This way we can prove that the family of plastic components attached
to these subsegments is finite too.
\end{rem}

\section{The Number of Plastic Components}

Next, we are going to give an estimate on the number of plastic components.
Let $\partial U=\lambda\cup\Lambda$ where 
\[
\lambda:=\bar{S}_{1}=\{(x_{1},\rho_{1}x_{1}+\rho_{2})\,\mid\,0\le x_{1}\le b\}
\]
and $\Lambda:=\bar{S}_{2}\cup\cdots\cup\bar{S}_{m}$. We also assume
that some $U$-neighborhood of $\lambda$ lies in $\{x_{2}>\rho_{1}x_{1}+\rho_{2}\}$. 

Let $y=f(s)$ be a parametrization of $\partial U$ for $0\le s\le L$,
with $f(0)=(b,\rho_{1}b+\rho_{2})=V_{0}$ and $f(s_{1})=(0,\rho_{2})=V_{1}$. 

We know that along $\partial U$ the $p$-distance function $d_{p}$
is differentiable, except at the points $f(s_{j})=V_{j}$. Let $\nu(s)=(\nu_{1}(s),\nu_{2}(s))$
for $s\ne s_{j}$ be the inward normal to $\partial U$ at $f(s)$
with $\gamma_{q}(\nu(s))=1$. Also let $\zeta=\frac{1}{\sqrt{1+\rho_{1}^{2}}}(1,\rho_{1})$
as before. Then by (\ref{eq: Dd_p}) and continuity of $Dd_{p}$,
we have 
\[
D_{\zeta}d_{p}(f(s))=\nu(s)\cdot\zeta.
\]

\begin{assumption}
\label{assu: 2}The set $\{s\in[0,L]-\{s_{j}\}\,\mid\,\nu(s)\cdot\zeta=0\}$
consists of a finite number of points, and a finite number of intervals.
\end{assumption}
Therefore $\nu\cdot\zeta$ changes sign a finite number of times.
Let 
\begin{eqnarray}
 &  & k:=\textrm{ The number of times }\nu\cdot\zeta\label{eq: k}\\
 &  & \textrm{changes sign from positive to negative on the interval }[s_{1},L].\nonumber 
\end{eqnarray}

Consider $f(s)+\delta(s)\mu(s)$ for $s\ne s_{j}$, which parametrizes
the free boundary when $\delta>0$. Note that $d_{p}$ is $C^{1,\alpha}$
around these points even if $\delta(s)=0$. Since $f(s)$ is the unique
$p$-closest point on $\partial U$ to $f(s)+\delta(s)\mu(s)$ when
$s\ne s_{j}$, by (\ref{eq: Dd_p}) we have 
\[
D_{\zeta}d_{p}(f(s)+\delta(s)\mu(s))=\nu(s)\cdot\zeta.
\]

Now consider the function 
\begin{eqnarray*}
 &  & u_{1}(s):=D_{\zeta}u(f(s)+\delta(s)\mu(s))\qquad s\ne s_{j}\\
 &  & u_{1}(s_{j}):=0.
\end{eqnarray*}
Note that $u_{1}$ is continuous at $s_{j}$'s. The reason is that
$Du(f(s_{j}))=0$ by continuity of $Du$ there, and the fact that
the directional derivatives of $u$ vanish in two directions at $f(s_{j})$.
\begin{lem}
\label{lem: signs}$u_{1}(s)$ has the same sign as $\nu(s)\cdot\zeta$
for $s\ne s_{j}$. \end{lem}
\begin{proof}
Since on the free boundary $Du=Dd_{p}$, we have 
\[
u_{1}(s)=\nu(s)\cdot\zeta
\]
when $\delta(s)>0$. 

Consider a point $s_{0}$ different than $s_{j}$'s, with $\nu(s_{0})\cdot\zeta>0$.
If $\delta(s_{0})>0$ then obviously $u_{1}(s_{0})>0$ too. If $\delta(s_{0})=0$
but $s_{0}=\lim s_{k}$ where $\delta(s_{k})>0$, then by continuity
we still have $u_{1}(s_{0})=\nu(s_{0})\cdot\zeta>0$. And finally,
if neither of these happen at $s_{0}$, then $\delta\equiv0$ on a
neighborhood of $s_{0}$. This means that some $U$-neighborhood of
$f(s_{0})$ is elastic. Thus in that neighborhood we have 
\[
-\Delta u=\eta>0.
\]
As $u>0$ in $U$ and $u=0$ on $\partial U$, the strong maximum
principle (actually the Hopf's lemma used in its proof) implies that
\begin{equation}
\nu(s_{0})\cdot\zeta\, D_{\zeta}u(f(s_{0}))+\nu(s_{0})\cdot\xi\, D_{\xi}u(f(s_{0}))=D_{\nu}u(f(s_{0}))>0.\label{eq: normal der}
\end{equation}
Here $\xi$ is a unit vector orthogonal to $\zeta$. On the other
hand, $u$ is constant along $\partial U$, therefore its tangential
derivative vanishes, i.e. 
\begin{equation}
-\nu(s_{0})\cdot\xi\, D_{\zeta}u(f(s_{0}))+\nu(s_{0})\cdot\zeta\, D_{\xi}u(f(s_{0}))=0.\label{eq: tangent der}
\end{equation}
Now using this and the fact that $\nu(s_{0})\cdot\zeta>0$, we can
rewrite (\ref{eq: normal der}) to get 
\[
[\nu(s_{0})\cdot\zeta+\frac{[\nu(s_{0})\cdot\xi]^{2}}{\nu(s_{0})\cdot\zeta}]\, D_{\zeta}u(f(s_{0}))>0.
\]
Hence $u_{1}(s_{0})=D_{\zeta}u(f(s_{0}))>0$ as desired. When $\nu(s_{0})\cdot\zeta<0$,
we can repeat the above arguments to deduce that $u_{1}(s_{0})<0$
too. 

When $\nu(s_{0})\cdot\zeta=0$, we can still deduce that $u_{1}(s_{0})=0$.
The only difference with the above argument is that when $\delta\equiv0$
on a neighborhood of $s_{0}$, we have to use (\ref{eq: tangent der})
to get the result, noting that $\nu(s_{0})\cdot\xi\ne0$ when $\nu(s_{0})\cdot\zeta=0$. 
\end{proof}
It should be noted that $u_{1}(s)=0$ for $s\in[0,s_{1}]$.
\begin{defn}
The points of the form $f(s)+\delta(s)\mu(s)$ for which $u_{1}(s)=0$
will be called \textbf{flat points}. By Assumption \ref{assu: 2}
and the above argument, the set of flat points consists of a finite
number of points, and a finite number of arcs called \textbf{flat
intervals}.
\end{defn}

Consider the harmonic function $D_{\zeta}u$ over the elastic region
$E$. $D_{\zeta}u$ has harmonic continuation to a neighborhood of
each interior point of a flat interval, if around that point either
$\delta>0$ or $\delta\equiv0$. The reason is that for a flat interval
we have $\nu\cdot\zeta\equiv0$ over the part of $\partial U$ attached
to it. Hence that part of $\partial U$ is a line segment in the $\zeta$
direction. Thus the flat interval which is either this line segment,
or a free boundary attached to it, is in both cases an analytic curve.
\begin{lem}
\label{lem: level curves}Let $x_{0}\in E$ be a point where $D_{\zeta}u(x_{0})=0$.
There exists a simple Jordan curve $\{t\mapsto\gamma(t)\,;\, t\in\mathbb{R}\}$
in $E$ passing through $x_{0}$, along which $D_{\zeta}u=0$. Furthermore,
\[
\underset{t\to-\infty}{\lim}\gamma(t)\qquad,\qquad\underset{t\to+\infty}{\lim}\gamma(t)
\]
exist, are different, and belong to $\partial E$.\end{lem}
\begin{proof}
Since $D_{\zeta}u$ is harmonic, its level sets in $E$ are locally,
the union of several analytic arcs emanating from a vertex. Consider
the family of injective continuous maps from $(-1,1)$ into the level
set of $D_{\zeta}u$ at $x_{0}$, which take zero to $x_{0}$. We
endow this family with a partial order relation. For $f_{1},f_{2}$
in the family, we say $f_{1}\le f_{2}$ if 
\[
f_{1}((-1,1))\subseteq f_{2}((-1,1)).
\]
Now, we can apply Zorn's lemma to deduce the existence of a maximal
map. We only need to check that any increasing chain has an upper
bound. Consider such a chain $\{f_{\alpha}\}$. We claim that each
$f_{\beta}((-1,1))$ is open in $\underset{\alpha}{\cup}f_{\alpha}((-1,1))$.
Consider a point $f_{\beta}(t_{0})$ in $f_{\beta}((-1,1))$, then
the level set around it, is the union of several arcs emanating from
it, and $f_{\beta}((t_{0}-\epsilon,t_{0}+\epsilon))$ is one of them.
Now, none of the sets $f_{\alpha}((-1,1))-f_{\beta}((t_{0}-\epsilon,t_{0}+\epsilon))$
can intersect one of these arcs. Since otherwise we have a loop in
the level set, which results in $D_{\zeta}u\equiv0$ by the maximum
principle and simple connectedness of $E$. This contradiction gives
the result.%
\footnote{This is a contradiction, because otherwise $u$ must be constant zero
on any segment in $E$ in the $\pm\zeta$ direction that starts from
a point on $\partial U$ close to the vertices of $\lambda$. This
means $u\equiv0$ on an open subset of $E$, contradicting the fact
that $\Delta u=-\eta\ne0$ there.%
}

Therefore $\underset{\alpha}{\cup}f_{\alpha}((-1,1))$ is the union
of countably many of $f_{\alpha}((-1,1))$'s, since the topology of
$\mathbb{R}^{2}$ is second countable. Now, by reparametrizing the
maps in this countable subchain and gluing them together, we obtain
a continuous map from $(-1,1)$ onto $\underset{\alpha}{\cup}f_{\alpha}((-1,1))$.
The injectivity of this map is easy to show, since if it fails it
must fail for one of the maps in the countable subchain too.

Now, consider $\gamma$, a maximal simple Jordan curve in the level
set $\{D_{\zeta}u=0\}$ passing through $x_{0}$, parametrized from
$-\infty$ to $\infty$ with $\gamma(0)=x_{0}$. Since $E$ is bounded,
every sequence $t_{k}\to\infty$ has a subsequence such that $\gamma(t_{k_{i}})\to x^{*}$.
If $x^{*}\in E$, then $\gamma(t_{k_{i}})$ belongs to one of the
arcs in the level set emanating from $x^{*}$. Thus, the tale of $\gamma$
coincides with that arc, as the level set around $x^{*}$ is the union
of those arcs, and $\gamma$ is one to one. Therefore, either $\gamma$
can be extended beyond $x^{*}$, or we get a loop in the level set,%
\footnote{Depending on wether $x^{*}$ is on the image of $\gamma$ or not.%
} which are contradictions. Hence, every such limit must belong to
$\partial E$ and be a flat point. 

Now suppose that for two sequences $t_{k},t_{l}\to\infty$, we have
$\gamma(t_{k})\to x^{*}$ and $\gamma(t_{l})\to x'$, where $x^{*},x'\in\partial E$.
Suppose $x^{*}\ne x'$ and one of them, say $x'$, belongs to the
interior of a flat interval. Then, if $D_{\zeta}u$ has harmonic continuation
in a disk around $x'$, the level set $\{D_{\zeta}u=0\}$ is again
the union of finitely many arcs emanating from $x'$. Therefore, $\gamma$
can not intersect the boundary of that disk an infinite number of
times, contradicting our assumption. If $D_{\zeta}u$ does not have
harmonic continuation around $x'$, then $\delta(x')=0$ and a sequence
of plastic components accumulate at $x'$. In this case, we can find
a sequence of points $\gamma(t_{l'})$ at an appropriate distance
from $\gamma(t_{l})$, such that $\gamma(t_{l'})\to x''$. Where $x''$
is in the interior of the same flat interval, and either $\delta(x'')>0$
or $\delta\equiv0$ around it. Thus $D_{\zeta}u$ has harmonic continuation
around $x''$ and we can argue as before.

Thus, if $x^{*}\ne x'$ then none of them can belong to the interior
of a flat interval. Hence they are either isolated flat points or
the endpoints of flat intervals. But again, looking at the arcs between
$\gamma(t_{k})$ and $\gamma(t_{l})$ on the image of $\gamma$, we
see that there are infinitely many limit points on $\partial E$ between
$x^{*},x'$, which contradicts Assumption \ref{assu: 2} and the argument
in the previous paragraph. Hence the limits $\underset{t\to\pm\infty}{\lim}\gamma(t)$
exist. 

Finally, if the two limit points of $\gamma$ coincide, the strong
maximum principle and continuity of $Du$ over $\bar{U}$ imply that
$D_{\zeta}u\equiv0$ over some domain, and consequently over $E$,
which is a contradiction.\end{proof}
\begin{rem}
Note that the endpoints of $\gamma$ can belong to the interior of
a flat interval. Also, $\gamma$ is analytic except at a countable
number of points. The reason is that its singularity can happen at
the points where $DD_{\zeta}u=0$. But $D_{\zeta}u$ is harmonic and
this set of points is at most countable with accumulation points on
$\partial E$. \end{rem}
\begin{lem}
\label{lem: finite level curves}The set of level curves in Lemma
\ref{lem: level curves} is finite.\end{lem}
\begin{proof}
First, note that any such curve can not have both its endpoints on
the same flat interval, since otherwise $D_{\zeta}u\equiv0$ on $E$
which is a contradiction. Second, for the same reason, two such curves
can not have the same endpoints, or have each of their endpoints on
the same flat intervals, or one endpoint the same and the other one
on one flat interval.

Therefore, there is at most one such curve, connecting two isolated
flat points, or two flat intervals, or an isolated flat point and
a flat interval. Hence we get the result.\end{proof}
\begin{rem}
A consequence of this lemma is that all the level curves given by
Lemma \ref{lem: level curves} are piecewise analytic. The reason
is that the singularities of the level curves happen at the zeros
of $DD_{\zeta}u$, and through any such point at least two level curves
pass. Thus their number must be finite, as no two level curves can
intersect more than once.
\end{rem}
Let us fix some notation before proceeding. We denote by $\lambda_{E}$
the part of $\lambda$ with no plastic component attached to it. We
also denote by $\Gamma_{\lambda}$ the union of the free boundaries
of the plastic components attached to $\lambda$. Finally let $\lambda_{0}:=\bar{\lambda}_{E}\cup\Gamma_{\lambda}$.
Similarly we define $\Lambda_{E},\Gamma_{\Lambda}$ and $\Lambda_{0}$.

Since all the corners of $\partial U$ are nonreentrant and they have
an elastic neighborhood, the number of plastic components attached
to $\lambda$ is finite by Theorem \ref{thm: plastic component}.
We denote these plastic components by 
\[
P_{j}:=\{(x_{1},\rho_{1}x_{1}+\rho_{2})+t(\mu_{1},\mu_{2})\,\mid\, a_{j}\le x_{1}\le b_{j}\,,\,0\le t\le\delta(x_{1})\}\qquad j=1,2,\cdots\tau,
\]
where $(\mu_{1},\mu_{2})$ is the inward $p$-normal, and 
\[
0<a_{1}<b_{1}\le a_{2}<b_{2}\le\cdots\le a_{\tau}<b_{\tau}<b.
\]
In each interval $\{a_{j}\le x_{1}\le b_{j}\}$, $\delta$ has $N_{j}$
local maxima. These are strict local maxima since the tilted graph
of $\delta$, which is the free boundary, is an analytic curve.

Consider one of the plastic components $P_{j}$. Let $\beta$ be a
point of local maximum of $\delta(x_{1})$ over $x_{1}\in[a_{j},b_{j}]$.
\begin{lem}
\label{lem: level curve from max}There exists a level curve $\{t\mapsto\gamma(t)\,;\, t\in\mathbb{R}\}$
of $\{D_{\zeta}u=0\}$ in $E$ with no self intersections, such that
\[
\underset{t\to-\infty}{\lim}\gamma(t)=(\beta,\rho_{1}\beta+\rho_{2})+\delta(\beta)(\mu_{1},\mu_{2})=:\tilde{\beta},
\]
and $\gamma(\infty):=\underset{t\to\infty}{\lim}\gamma(t)$ belongs
to $\Lambda_{0}-\lambda_{0}$.\end{lem}
\begin{proof}
The fact that $\gamma(\infty)$ can not belong to $\lambda_{0}$,
or $\gamma$ does not intersect itself, is a consequence of the strong
maximum principle as argued before. Now let us show the existence
of such a level curve. When $\epsilon>0$ is small enough, as $\beta$
is a strict local maximum, we have 
\[
(\beta\pm\epsilon,\rho_{1}(\beta\pm\epsilon)+\rho_{2})+\delta(\beta)(\mu_{1},\mu_{2})\in E.
\]
Thus the line 
\[
t\mapsto(\beta+\delta(\beta)\mu_{1}+t\,,\,\rho_{1}\beta+\rho_{2}+\delta(\beta)\mu_{2}+\rho_{1}t)
\]
is tangent to the free boundary at $\tilde{\beta}$. Hence the unit
vector tangent to the free boundary at $\tilde{\beta}$ is $\zeta$. 

Now, since $Du=Dd_{p}=\nu$ on the free boundary, we have $D_{\nu}u=|\nu|^{2}$
there. As $\nu$ is constant along $\lambda$, the derivative of $D_{\nu}u$
vanishes along the free boundary containing $\tilde{\beta}$. The
same is true about the derivative of $D_{\zeta}u$ along that part
of the free boundary, as $D_{\zeta}u$ is constant zero there. Therefore
we have 
\begin{eqnarray*}
 &  & D_{\zeta}D_{\zeta}u(\tilde{\beta})=0\\
 &  & D_{\zeta}D_{\nu}u(\tilde{\beta})=0.
\end{eqnarray*}
Note that as this part of the free boundary is smooth, $u$ is also
smooth along it, since it equals the smooth function $d_{p}$ along
it, and satisfies $\Delta u=-\eta$ in $E$. Consequently, $D^{2}u$
converges the correct limit as we approach the free boundary through
points inside $E$.

But $D_{\zeta}u$ has harmonic continuation in a neighborhood of $\tilde{\beta}$,
so if $x_{i}\in E$ converge to $\tilde{\beta}$, we have 
\[
D_{\nu}D_{\zeta}u(\tilde{\beta})=\lim D_{\nu}D_{\zeta}u(x_{i})=\lim D_{\zeta}D_{\nu}u(x_{i})=D_{\zeta}D_{\nu}u(\tilde{\beta})=0.
\]
Thus $DD_{\zeta}u(\tilde{\beta})=0$. Hence the level set of $D_{\zeta}u$
at $\tilde{\beta}$ must be the union of at least four arcs emanating
from $\tilde{\beta}$ making equal angles with each other. Thus, there
is at least one level curve starting at $\tilde{\beta}$ that remains
in $E$. Now, similarly to Lemma \ref{lem: level curves}, we can
extend this level curve until it hits $\partial E$.
\end{proof}

\begin{rem}
The conclusion of the above lemma is also true when $\beta$ is a
point of local minimum with $\delta(\beta)>0$.
\end{rem}
Now we state our main result in this section. Remember that $k$ is
given by (\ref{eq: k}).
\begin{thm}
Each $N_{j}$ is finite and 
\[
N:=\overset{\tau}{\underset{j=1}{\sum}}N_{j}\le k.
\]
\end{thm}
\begin{proof}
Note that Lemmata \ref{lem: finite level curves} and \ref{lem: level curve from max}
imply that each $N_{j}$ is finite. Because no level curve can have
both its endpoints on $\lambda_{0}$, as otherwise we have $D_{\zeta}u\equiv0$. 

Now consider the finite set of level curves $\tilde{\gamma}_{i}$
given by Lemma \ref{lem: level curves}, that have both their endpoints
on $\Lambda_{0}$. Let $\hat{\gamma}_{j}$'s be the parts of the other
level curves that have both endpoints on $\tilde{\gamma}_{i}$'s,
or one endpoint at them and the other one on $\Lambda_{0}$. Note
that two level curves can not intersect at more than one point. Thus
the number of $\hat{\gamma}_{j}$'s is finite. Also note that two
level curves with one endpoint on $\lambda_{0}$ can not intersect.

Denote by $E_{1}$ the component of $E-\{\tilde{\gamma}_{i},\hat{\gamma}_{j}\}$
which is attached to $\lambda_{0}$. The boundary of $E_{1}$ consists
of $\lambda_{0}$ and part of $\Lambda_{0}$ together with parts of
some $\tilde{\gamma}_{i}$'s and $\hat{\gamma}_{j}$'s. Let 
\[
\Lambda_{1}:=\overline{\partial E_{1}-\lambda_{0}}.
\]
Note that by our construction, any level curve in $E_{1}$ given by
Lemma \ref{lem: level curves} must have one endpoint on $\lambda_{0}$.
Let $\gamma_{1},\cdots\gamma_{N}$ be the level curves given by Lemma
\ref{lem: level curve from max}, numbered as we move from $V_{1}$
to $V_{0}$. Then one endpoint of each $\gamma_{i}$ is a strict local
maximum point on the tilted graph of $\delta$ over $\lambda$ which
we call it $\beta_{i}$, and the other endpoint is on $\Lambda_{1}$
which we call it $\tau_{i}$. Let $\mathcal{D}_{1},\cdots,\mathcal{D}_{N+1}$
be the components of $E_{1}-\{\gamma_{i}\}$. Note that $\bar{\gamma}_{i}\cap\bar{\gamma}_{j}=\emptyset$
when $i\ne j$.

Consider $\mathcal{D}_{i}$, whose boundary consists of $\gamma_{i-1},\gamma_{i}$
and parts of $\lambda_{0},\Lambda_{1}$, which we denote the latter
two by $\lambda_{0i},\Lambda_{1i}$. Note that $\gamma_{0},\gamma_{N+1}$
are empty. Suppose $1<i<N+1$. First we claim that $D_{\zeta}u$ must
change sign along $\Lambda_{1i}$. Otherwise we have for example $D_{\zeta}u\ge0$
there. As $D_{\zeta}u$ vanishes on the other parts of $\partial\mathcal{D}_{i}$,
maximum principle implies 
\[
D_{\zeta}u>0\qquad\textrm{ in }\mathcal{D}_{i}.
\]
Since near $\lambda_{0}$ we have $D_{\zeta}d_{p}=0$, we get 
\[
D_{\zeta}(d_{p}-u)<0
\]
near $\lambda_{0}$ in $\mathcal{D}_{i}$. This implies that $\delta$
is strictly increasing along the subset of $\lambda_{0i}$ over which
$\delta>0$. To see this, just look at the behavior of $d_{p}-u$
on segments in the $\zeta$ direction starting on the free boundary.
Hence we get a contradiction with $\beta_{i-1}$ being a strict local
maximum. Note that this argument also shows that $D_{\zeta}u$ must
be positive on part of $\Lambda_{11}$, and negative on part of $\Lambda_{1N+1}$.

Let $i\ne1,N+1$. Then consider the finite set of level curves of
$D_{\zeta}u=0$ in $\mathcal{D}_{i}$. These level curves have one
endpoint on $\lambda_{0i}$ and one endpoint on $\Lambda_{1i}$, and
do not intersect each other. Consider the one closest to $\gamma_{i-1}$,
and let $\tilde{\mathcal{D}}$ be the subdomain of $\mathcal{D}_{i}$
that they enclose. Then $D_{\zeta}u$ must have one sign on $\tilde{\mathcal{D}}$,
since it can not vanish there, as there is no further level curve
inside $\tilde{\mathcal{D}}$. Thus we must have $D_{\zeta}u<0$ on
$\tilde{\mathcal{D}}$. Since otherwise we get as before that $\delta$
is strictly increasing near and on the right of $\beta_{i-1}$, contradicting
the fact that it is a local maximum. Hence $D_{\zeta}u$ must be negative
on some part of $\Lambda_{1}$ near and on the right of $\tau_{i-1}$.
Similarly, $D_{\zeta}u$ must be positive on some part of $\Lambda_{1}$
near and on the left of $\tau_{i}$.

Therefore, $D_{\zeta}u$ must change sign from positive to negative
along $\Lambda_{1}$ at least $N$ times. Finally note that as $D_{\zeta}u$
vanishes on $\tilde{\gamma}_{i}$'s and $\hat{\gamma}_{j}$'s, these
sign changes are actually happening along $\Lambda_{0}$. Thus by
Lemma \ref{lem: signs} we get the desired result.
\end{proof}
We immediately get the following
\begin{thm}
\label{thm: polygons}Suppose $U$ is a convex polygon. Then for any
side $S_{j}$ there is at most one plastic loop attached to it. Furthermore,
the plastic loop is given by 
\[
\{f(s)+t\mu(s)\,\mid\, s\in(a_{j},b_{j})\,,\, t\in(0,\delta(s))\},
\]
where $s_{j-1}<a_{j}<b_{j}<s_{j}$. Also there is $c_{j}\in(a_{j},b_{j})$
such that $\delta(s)$ is strictly increasing for $s\in(a_{j},c_{j})$
and strictly decreasing for $s\in(c_{j},b_{j})$.\end{thm}
\begin{proof}
Let $\zeta_{j}$ be the unit vector in the $S_{j}$ direction. We
only need to notice that since $U$ is a convex polygon, $\nu\cdot\zeta_{j}$
is zero on at most one $S_{i}$ for $i\ne j$. Thus it changes sign
from positive to negative exactly once, and we have $k=1$. The second
part of the theorem follows from analyticity of the tilted graph of
$\delta$.
\end{proof}

\section{Reflection Method}

In this section we give an example of how to apply the reflection
method in \citet{MR563207} to our problem. Let $U$ be the rectangle
\[
\{(x_{1},x_{2})\,\mid\,|x_{1}|<a\,,\,|x_{2}|<b\}.
\]
By Theorem \ref{thm: polygons}, symmetry of $\gamma_{p}$, and symmetry
of $U$, there are four plastic components 
\begin{eqnarray*}
 &  & P_{1}:\;|x_{1}|\le\alpha\;,\;-b\le x_{2}\le-b+\phi(x_{1}),\\
 &  & P_{2}:\;|x_{2}|\le\beta\;,\;-a\le x_{1}\le-a+\psi(x_{1}),\\
 &  & P_{3}:\;\textrm{the reflection of \ensuremath{P_{1}}with respect to the \ensuremath{x_{1}}axis},\\
 &  & P_{4}:\;\textrm{the reflection of \ensuremath{P_{2}}with respect to the \ensuremath{x_{2}}axis}.
\end{eqnarray*}
Here $\phi,\psi$ are even functions. Let $\rho$ be the reflection
with respect to the bisector of $\partial U$ at $(-a,-b)$, i.e.
\[
x_{2}=x_{1}+a-b.
\]
Thus 
\[
\rho(x_{1},x_{2})=(x_{2}-a+b,x_{1}+a-b).
\]

\begin{thm}
If $b<a$, then $\rho(P_{2})\subset P_{1}$.\end{thm}
\begin{proof}
Let 
\[
\mathcal{D}:=E\cap\{(x_{1},x_{2})\,\mid\, x_{2}<x_{1}+a-b\;,\;-a<x_{1}<-a+2b\}.
\]
Consider the function 
\[
w(x):=u(\rho(x))-u(x)
\]
in $\mathcal{D}$. Since $\Delta u\ge-\eta$, and $\Delta u(x)=-\eta$
for $x\in E$, we have 
\[
\Delta w\ge0
\]
in $\mathcal{D}$, noting that Laplacian is invariant under reflections.

The boundary of $\mathcal{\mathcal{D}}$ consists of parts of the
lines $x_{2}=-b$, $x_{2}=x_{1}+a-b$, $x_{1}=-a+2b$, and parts of
$\Gamma_{1},\Gamma_{3},\Gamma_{4}$. Here $\Gamma_{i}$ is the free
boundary attached to $P_{i}$. Note that some of these parts can be
empty. Also note that $\Gamma_{2}$ is on the other side of the line
$x_{2}=x_{1}+a-b$, so it does not intersect $\partial\mathcal{\mathcal{D}}$.

Since $u$ vanishes on $\partial U$, and $\rho$ takes $x_{2}=-b$
to $x_{1}=-a$, $w=0$ on it. The same is true on the line $x_{2}=x_{1}+a-b$,
as it is fixed by $\rho$. Also as $\rho$ takes $x_{1}=-a+2b$ to
$x_{2}=b$, for $x$ on it we have 
\[
w(x)=0-u(x)\le0.
\]
If $x\in\Gamma_{1}$ then $u(x)=d_{p}(x)$. But $d_{p}(\rho(x))\le d_{p}(x)$,
since due to the symmetry of $\gamma_{p}$, $\rho(x)$ has the same
$p$-distance to $x_{1}=-a$ as $x$ has to $x_{2}=-b$. Thus 
\[
w(x)=u(\rho(x))-d_{p}(x)\le u(\rho(x))-d_{p}(\rho(x))\le0.
\]
We can argue similarly when $x\in\Gamma_{3}$, noting that $\rho$
decreases the $p$-distance to $x_{2}=b$ over $\mathcal{D}$. Finally
when $x\in\Gamma_{4}$, we get the same result noting that the $p$-distance
of $\rho(x)$ to $x_{2}=b$ is less than the $p$-distance of $x$
to $x_{1}=a$, when $x\in\mathcal{D}$.

Therefore, by the strong maximum principle 
\begin{eqnarray*}
w(x)<0 &  & x\in\mathcal{D}.
\end{eqnarray*}
Note that if $w\equiv0$ on $\bar{\mathcal{D}}$, then we must have
$u=0$ on $x_{1}=-a+2b$ inside $U$, which is impossible.

Now suppose there is $x\in P_{2}$ such that $\rho(x)\notin P_{1}$.
Then $\rho(x)\in\mathcal{D}$. Thus 
\[
w(\rho(x))=u(x)-u(\rho(x))=d_{p}(x)-u(\rho(x)).
\]
But $d_{p}(\rho(x))\le d_{p}(x)$ as the $p$-distance of $\rho(x)$
to $x_{2}=-b$ equals the $p$-distance of $x$ to $x_{1}=-a$. Hence
\[
0>w(\rho(x))\ge d_{p}(\rho(x))-u(\rho(x)),
\]
which contradicts $u\le d_{p}$.\end{proof}
\begin{rem}
Since $\gamma_{p}$ is not invariant under arbitrary reflections,
the more general results proved in \citet{MR563207} using reflections
does not necessarily hold here. Although some special cases can be
proved similar to the above, for example when a bisector of a triangle
is parallel to one of the coordinate axes.
\end{rem}

\bibliographystyle{plainnat}
\bibliography{New-Paper-Bibliography}

\end{document}